\documentclass[11 pt]{amsart}
\usepackage{latexsym,amscd,amssymb, graphicx, amsthm}  
\usepackage{young}
\usepackage{tikz}

\usepackage[margin=1in]{geometry}

\numberwithin{equation}{section}

\newtheorem{theorem}{Theorem}[section]
\newtheorem{proposition}[theorem]{Proposition}
\newtheorem{corollary}[theorem]{Corollary}
\newtheorem{lemma}[theorem]{Lemma}
\newtheorem{conjecture}[theorem]{Conjecture}

\newtheorem{question}[theorem]{Question}

\newtheorem{defn}[theorem]{Definition}
\theoremstyle{definition}

\newcommand{\maj}{{\mathrm {maj}}}
\newcommand{\inv}{{\mathrm {inv}}}

\newcommand{\Des}{{\mathrm {Des}}}

\newcommand{\Hilb}{{\mathrm {Hilb}}}
\newcommand{\grFrob}{{\mathrm {grFrob}}}

\newcommand{\Ch}{{\mathrm {Ch}}}

\newcommand{\NSym}{{\mathbf {NSym}}}
\newcommand{\QSym}{{\mathrm {QSym}}}

\newcommand{\SYT}{{\mathrm {SYT}}}
\newcommand{\ch}{{\mathbf {ch}}}
\newcommand{\Frob}{{\mathrm {Frob}}}

\newcommand{\initial}{{\mathrm {in}}}

\newcommand{\symm}{{\mathfrak{S}}}

\newcommand{\wt}{{\mathrm{wt}}}

\newcommand{\FF}{{\mathbb {F}}}
\newcommand{\CC}{{\mathbb {C}}}
\newcommand{\QQ}{{\mathbb {Q}}}
\newcommand{\ZZ}{{\mathbb {Z}}}

\newcommand{\NNN}{{\mathcal{N}}}

\newcommand{\CCC}{{\mathcal{C}}}

\newcommand{\MMM}{{\mathcal{M}}}
\newcommand{\BBB}{{\mathcal{B}}}

\newcommand{\xx}{{\mathbf {x}}}
\newcommand{\II}{{\mathbf {I}}}
\newcommand{\yy}{{\mathbf {y}}}
\newcommand{\TT}{{\mathbf {T}}}

\newcommand{\sss}{{\mathbf {s}}}

\newcommand{\rev}{{\mathrm {rev}}}

\begin{document}

\title[Tail positive words and generalized coinvariant algebras]
{Tail positive words and generalized coinvariant algebras}

\author{Brendon Rhoades}
\address
{Department of Mathematics \newline \indent
University of California, San Diego \newline \indent
La Jolla, CA, 92093-0112, USA}
\email{bprhoades@math.ucsd.edu}

\author{Andrew Timothy Wilson}
\address
{Department of Mathematics \newline \indent
University of Pennsylvania \newline \indent
Philadelphia, PA, 19104-6395, USA}
\email{andwils@math.upenn.edu}

\begin{abstract}
Let $n,k,$ and $r$ be nonnegative integers and let $S_n$ be the symmetric group. 
We introduce a quotient $R_{n,k,r}$ of the polynomial ring
$\QQ[x_1, \dots, x_n]$ in $n$ variables which carries the structure of a graded $S_n$-module.  
When $r \geq n$ or $k = 0$ the quotient $R_{n,k,r}$ reduces to the classical coinvariant algebra $R_n$
attached to the symmetric group.  
Just as algebraic properties of $R_n$ are controlled by combinatorial properties of permutations in $S_n$,
the algebra of $R_{n,k,r}$ is controlled by the combinatorics of objects called {\em tail positive words}.
We calculate the standard monomial basis of $R_{n,k,r}$ and 
its graded $S_n$-isomorphism type.  
We also view $R_{n,k,r}$ as a module over the 0-Hecke algebra $H_n(0)$,
prove that $R_{n,k,r}$ is a projective 0-Hecke module,
and calculate its quasisymmetric and nonsymmetric 0-Hecke characteristics.
We conjecture a relationship between our quotient
$R_{n,k,r}$ and the delta operators of the theory of Macdonald polynomials.
\end{abstract}

\maketitle

\section{Introduction}
\label{Introduction}

Consider the action of the symmetric group $S_n$ on $n$ letters on the polynomial ring
$\QQ[\xx_n] := \QQ[x_1, \dots, x_n]$ given by variable permutation.  The polynomials belonging to
the invariant subring  
\begin{equation}
\QQ[\xx_n]^{S_n} := \{ f(\xx_n) \in \QQ[\xx_n] \,:\, \pi.f(\xx_n) = f(\xx_n) \text{ for all $\pi \in S_n$} \}
\end{equation}
are the {\em symmetric polynomials} in the variable set $\xx_n$.
Let $e_d(\xx_n)$ be the {\em elementary symmetric function} of degree $d$, that is
$e_d(\xx_n) = \sum_{1 \leq i_1 < \cdots < i_d \leq n} x_{i_1} \cdots x_{i_d}$.
It is well known that the set $\{e_1(\xx_n), \dots, e_n(\xx_n) \}$ gives an algebraically independent
homogeneous collection of generators for the ring $\QQ[\xx_n]^{S_n}$.

Let $\QQ[\xx_n]^{S_n}_+ \subset \QQ[\xx_n]^{S_n}$ be the subspace of symmetric polynomials
with vanishing constant term.  The {\em invariant ideal} $I_n \subseteq \QQ[\xx_n]$ is the ideal
\begin{equation}
I_n := \langle \QQ[\xx_n]^{S_n}_+ \rangle = \langle e_1(\xx_n), \dots, e_n(\xx_n) \rangle
\end{equation}
generated by this subspace.  The {\em coinvariant algebra} $R_n$ is the corresponding quotient:
\begin{equation}
R_n := \QQ[\xx_n]/I_n.
\end{equation}
The algebra $R_n$ is a graded $S_n$-module.

The coinvariant algebra is among the most important representations in algebraic combinatorics;  algebraic 
properties of $R_n$ are deeply tied to combinatorial properties of permutations in $S_n$.
E. Artin proved \cite{Artin} 
that the collection of `sub-staircase' monomials $\{ x_1^{i_1} \cdots x_n^{i_n} \,:\, 0 \leq i_j < j \}$
descends to a vector space basis for $R_n$, so that the Hilbert series of $R_n$ is given by
\begin{equation}
\mathrm{Hilb}(R_n;q) = (1+q)(1+q+q^2) \cdots (1+q+\cdots+q^{n-1}) = [n]!_q,
\end{equation}
the standard $q$-analog of $n!$.
Chevalley \cite{C}
 proved that as an {\em ungraded} $S_n$-module, we have $R_n \cong \QQ[S_n]$,
the regular representation of $S_n$.  Lusztig (unpublished)
and Stanley \cite{Stanley} refined this result to describe the 
{\em graded} isomorphism type of $R_n$ in terms of the major index statistic on standard Young tableaux.

In this paper we will study the following generalization of the coinvariant algebra $R_n$.  
Recall that the degree $d$ {\em homogeneous symmetric function} in $\QQ[\xx_n]$ is given by
$h_d(\xx_n) := \sum_{1 \leq i_1 \leq \cdots \leq i_d \leq n} x_{i_1} \cdots x_{i_d}$.

\begin{defn}
\label{main-definition}
Let $n,k,$ and $r$ be nonnegative integers with $r \leq n$.
Let $I_{n,k,r} \subseteq \QQ[\xx_n]$ be the ideal
$$
I_{n,k,r} := \langle h_{k+1}(\xx_n), h_{k+2}(\xx_n), \dots, h_{k+n}(\xx_n), e_n(\xx_n), e_{n-1}(\xx_n), \dots, e_{n-r+1}(\xx_n)
\rangle
$$
and let 
$$
R_{n,k,r} := \QQ[\xx_n]/I_{n,k,r}
$$
be the corresponding quotient ring.
\end{defn}

The ideal $I_{n,k,r}$ is homogeneous and stable under the action of the symmetric group, so that 
$R_{n,k,r}$ is a graded $S_n$-module.   Since the generators of $I_{n,k,r}$ are symmetric polynomials, we have 
the containment of ideals $I_{n,k,r} \subseteq I_n$, so that $R_{n,k,r}$ projects onto the classical
coinvariant algebra $R_n$.  If 
$r \geq n$
or $k = 0$ we have the equality $I_{n,k,r} = I_n$, so that 
$R_{n,k,r} = R_n$.

Just as  algebraic properties of $R_n$ are controlled by  combinatorics of permutations 
$\pi_1 \dots \pi_n$
of the set $\{1, 2, \dots, n\}$,
algebraic properties of $R_{n,k,r}$ will be controlled by the combinatorics of permutations 
$\pi_1 \dots \pi_{n+k}$
of 
the {\em multiset} $\{0^k, 1, 2, \dots, n\}$ whose last $r$ entries $\pi_{n+k-r+1} \dots \pi_{n+k-1} \pi_{n+k}$ are 
all nonzero.  Thinking of positive letters as weights, we will call such permutations {\em $r$-tail positive}.

Let $S_{n,k,r}$ be the collection of all $r$-tail positive permutations of the multiset 
$\{0^k, 1, 2, \dots, n\}$.  
For example, we have
\begin{equation*}
S_{2,2,1} = \{0012, 0021, 0102, 0201, 1002, 2001\}.
\end{equation*}
By considering the possible locations of the $k$ $0$'s in an element of $S_{n,k,r}$, it is immediate that
\begin{equation}
\label{tail-positive-count}
|S_{n,k,r}| = {n+k-r \choose k} \cdot |S_n| =  {n+k-r \choose k} \cdot n!.
\end{equation}
The basic enumeration of Equation~\ref{tail-positive-count} will manifest in  Hilbert series as
\begin{equation}
\label{hilbert-series-formula}
\Hilb(R_{n,k,r} ;q) = {n + k - r \brack k}_q \cdot \Hilb(R_n; q) = {n + k - r \brack k}_q \cdot [n]!_q,
\end{equation}
where ${m \brack i}_q := \frac{[m]!_q}{[i]!_q [m-i]!_q}$ is the usual $q$-binomial coefficient.  Going even further,
we have the following graded Frobenius image
\begin{equation}
\label{frobenius-series-formula}
\grFrob(R_{n,k,r};q) = {n + k - r \brack k}_q \cdot \grFrob(R_n; q) = 
{n + k - r \brack k}_q \cdot \sum_{T \in \mathrm{SYT}(n)} s_{\mathrm{shape}(T)},
\end{equation}
which implies that the quotient $R_{n,k,r}$ consists of ${n+k-r \choose k}$ copies of the coinvariant algebra $R_n$,
with grading shifts given by a $q$-binomial coefficient.
The authors know of no direct way to see this from Definition~\ref{main-definition}.

The ideal $I_{n,k,r}$ defining the quotient $R_{n,k,r}$ is of `mixed' type -- its generators come in two flavors:
the homogeneous symmetric functions $h_{k+1}(\xx_n), h_{k+2}(\xx_n), \dots, h_{k+n}(\xx_n)$
and the elementary symmetric functions $e_n(\xx_n), e_{n-1}(\xx_n), \dots, e_{n-r+1}(\xx_n)$.
Several mixed ideals have recently been introduced to give combinatorial generalizations of the coinvariant algebra.
\begin{itemize}
\item  Let $k \leq n$.  Haglund, Rhoades, and Shimozono \cite{HRS} studied the quotient of $\QQ[\xx_n]$ by the ideal
\begin{equation}
\langle x_1^k, x_2^k, \dots, x_n^k, e_n(\xx_n), e_{n-1}(\xx_n), \dots, e_{n-k+1}(\xx_n) \rangle.
\end{equation}
The generators  are high degree $S_n$-invariants $e_n(\xx_n), \dots, e_{n-k+1}(\xx_n)$ together with 
a homogeneous system of parameters $x_1^k, \dots, x_n^k$ of degree $k$ carrying the 
defining representation of $S_n$.  Algebraic properties of the corresponding quotient are controlled by
 combinatorial properties of $k$-block ordered set partitions of $\{1, 2, \dots, n\}$.
\item  Let $r \geq 2$ and let $\ZZ_r \wr S_n$ be the group of $n \times n$ monomial matrices whose
nonzero entries
are $r^{th}$ complex roots of unity (this is the group of `$r$-colored permutations' of $\{1, 2, \dots, n\}$).
Let $k \leq n$ be non-negative integers.  Chan and Rhoades \cite{CR} studied the quotient of $\CC[\xx_n]$ by the ideal
\begin{equation}
\langle x_1^{kr+1}, x_2^{kr+1}, \dots, x_n^{kr+1}, e_n(\xx_n^r), e_{n-1}(\xx_n^r), \dots, e_{n-k+1}(\xx_n^r) \rangle,
\end{equation}
where $f(\xx_n^r) = f(x_1^r, \dots, x_n^r)$ for any polynomial $f$.  The generators here are high degree
$\ZZ_r \wr S_n$-invariants $e_n(\xx_n^r), \dots, e_{n-k+1}(\xx_n^r)$ together with a h.s.o.p. 
$x_1^{kr+1}, \dots, x_n^{kr+1}$ of degree $kr+1$ carrying the dual of
the defining representation of $\ZZ_r \wr S_n$.  Algebraic properties of the corresponding quotient are controlled
by $k$-dimensional faces in the Coxeter complex attached to $\ZZ_r \wr S_n$. 
\item Let $\FF$ be any field and let $H_n(0)$ be the 0-Hecke algebra over $\FF$ of rank $n$; the algebra 
$H_n(0)$ acts on the polynomial ring $\FF[\xx_n]$ by isobaric divided difference operators.  
Let $k \leq n$ be positive integers.
Huang and Rhoades \cite{HR} studied the quotient of $\FF[\xx_n]$ by the ideal
\begin{equation}
\langle h_k(x_1), h_k(x_1, x_2), \dots, h_k(x_1, x_2, \dots, x_n), e_n(\xx_n), e_{n-1}(\xx_n), \dots, e_{n-k+1}(\xx_n) \rangle.
\end{equation}
Once again, the generators consist of high degree 
$H_n(0)$-invariants $e_n(\xx_n),  \dots, e_{n-k+1}(\xx_n)$ together with a h.s.o.p.
of degree $k$ carrying the defining representation of $H_n(0)$.
Algebraic properties of the quotient are controlled by 0-Hecke combinatorics of $k$-block ordered set 
partitions of $\{1, \dots, n\}$.
\end{itemize}
The novelty of this paper is that our mixed ideals consist of high degree invariants of different kinds: 
elementary and homogeneous.  It would be interesting to develop a more unified picture of 
the algebraic and combinatorial properties of mixed quotients of polynomial rings.

Our analysis of the rings $R_{n,k,r}$ will share many properties with the analyses of the previously
mentioned mixed quotients.  Since $I_{n,k,r}$ is not cut out by a  regular sequence of homogeneous
polynomials in $\QQ[\xx_n]$,
the usual commutative algebra tools (e.g. the Koszul complex) used to study the coinvariant algebra $R_n$
are unavailable to us.
These will be replaced by {\em combinatorial} commutative algebra tools (e.g. Gr\"obner theory).
We will see that the ideal $I_{n,k,r}$ has an explicit minimal Gr\"obner basis (with respect to the 
lexicographic term order) in terms of 
Demazure characters.  This Gr\"obner basis will yield the Hilbert series of $R_{n,k,r}$, as well
as an identification of its standard monomial basis.
The graded $S_n$-isomorphism type of $R_{n,k,r}$ will then be obtainable by 
constructing an appropriate short exact sequence to serve as a recursion.

The rest of the paper is organized as follows. 
In {\bf Section~\ref{Background}} we give background related to symmetric functions and Gr\"obner bases.
In {\bf Section~\ref{Hilbert}} we determine the Hilbert series of $R_{n,k,r}$ and calculate the 
standard monomial basis for $R_{n,k,r}$ with respect to the lexicographic term order.
In {\bf Section~\ref{Frobenius}} we determine the graded $S_n$-isomorphism type of $R_{n,k,r}$.
We also view $R_{n,k,r}$ as a module over the 0-Hecke algebra $H_n(0)$ and calculate its 
graded nonsymmetric and bigraded quasisymmetric 0-Hecke characteristics.
We close in {\bf Section~\ref{Open}} with some open problems.

\section{Background}
\label{Background}

\subsection{Words, partitions, and tableaux}
Let $w = w_1 \dots w_n$ be a word in the alphabet of nonnegative integers.
An index $1 \leq i \leq n-1$ is a {\em descent} of $w$ if $w_i > w_{i+1}$.  
The {\em descent set} of $w$ 
is $\Des(w) := \{1 \leq i \leq n-1 \,:\, w_i > w_{i+1} \}$ and the {\em major index}
of $w$ is $\maj(w) := \sum_{i \in \Des(w)} i$.  A pair of indices $1 \leq i < j \leq n$ is called an
{\em inversion} of $w$ if $w_i > w_j$; the {\em inversion number} $\inv(w)$ counts the inversions of $w$.
The word $w$ is called {\em $r$-tail positive} if its last $r$ letters $w_{n-r+1} \dots w_{n-1} w_n$ are positive.

Let $n \in \ZZ_{\geq 0}$.  A {\em partition} $\lambda$ of $n$ is a weakly decreasing sequence
$\lambda = (\lambda_1 \geq \cdots \geq \lambda_k)$ of positive integers with
$\lambda_1 + \cdots + \lambda_k = n$.  We write $\lambda \vdash n$ or $|\lambda| = n$ to indicate
 that $\lambda$ is a partition of $n$.
 The {\em Ferrers diagram} of $\lambda$ (in English notation) consists of $\lambda_i$ left-justified boxes in row
 $i$.  The Ferrers diagram of $(4,2,2) \vdash 8$ is shown below on the left.
 
 \begin{center}
 \begin{small}
 \begin{Young}
   & & & \cr
    & \cr
    &
 \end{Young}  \hspace{0.4in}
 \begin{Young}
 1 & 1 & 2 & 5 \cr
 2 & 3 \cr
 3 & 5
 \end{Young} \hspace{0.4in}
 \begin{Young}
 1 & 2 & 5 & 8 \cr
 3 & 4 \cr
 6 & 7
 \end{Young}
 \end{small}
 \end{center}

 Let $\lambda \vdash n$.  A {\em tableau} $T$ of shape $\lambda$ is a filling of the Ferrers diagram of $\lambda$
 with positive integers.  The tableau $T$ is called {\em semistandard} if its entires increase weakly across rows
 and strictly down columns.  The tableau $T$ is a {\em standard Young tableau} if it is semistandard and its
 entries consist of $1, 2, \dots, n$.  
 The tableau in the center above is semistandard and the tableau on the right above is standard.
 We let $\mathrm{shape}(T) = \lambda$ denote the {\em shape} of $T$ and
 let $\SYT(n)$ denote the collection of all standard Young tableaux with $n$ boxes.
 
 Given a standard tableau $T \in \SYT(n)$, an index $1 \leq i \leq n-1$ is a {\em descent} of $T$ if $i+1$
 appears in a lower row of $T$ than $i$.  Let $\Des(T)$ denote the set of descents of $T$ and let 
 $\maj(T) := \sum_{i \in \Des(T)} i$ be the {\em major index} of $T$.
 If $T$ is the standard tableau above we have $\Des(T) = \{2,5\}$ so that $\maj(T) = 2 + 5 = 7$.
 
 \subsection{Symmetric functions}
 Let $\Lambda$ denote the ring of symmetric functions in an infinite variable set $\xx = (x_1, x_2, \dots )$
 over the ground field $\QQ(q,t)$.
 The algebra $\Lambda$ is graded by degree: $\Lambda = \bigoplus_{n \geq 0} \Lambda_n$.
 The graded piece $\Lambda_n$ has dimension equal to the number of partitions $\lambda \vdash n$.
 
 The vector space $\Lambda_n$ has many interesting bases, all indexed by partitions of $n$.  
 Given $\lambda \vdash n$, let
 \begin{equation*}
 m_{\lambda}, \hspace{0.1in}
 e_{\lambda},  \hspace{0.1in}
 h_{\lambda},  \hspace{0.1in}
 p_{\lambda},  \hspace{0.1in}
 s_{\lambda},  \hspace{0.1in}
 \widetilde{H}_{\lambda}
 \end{equation*}
 denote the associated {\em monomial, elementary, homogeneous, power sum, Schur,} and 
 {\em modified Macdonald} symmetric function (respectively).
As $\lambda$ ranges over the collection of partitions of $n$, all of these form bases for the vector space 
$\Lambda_n$.

Let $f(\xx) \in \Lambda$ be a symmetric function.  We define an eigenoperator
$\Delta_f: \Lambda \rightarrow \Lambda$ for the modified Macdonald basis of $\Lambda$
as follows.
Given a partition $\lambda$, we set
\begin{equation}
\Delta_f(\widetilde{H}_{\lambda}) := f( \dots, q^{i-1} t^{j-1}, \dots ) \cdot \widetilde{H}_{\lambda},
\end{equation}
where $(i,j)$ ranges over all matrix coordinates of cells in the Ferrers diagram of $\lambda$.  
The reader familiar with plethysm will recognize this formula as
\begin{equation}
\Delta_f(\widetilde{H}_{\lambda}) := f[B_{\lambda}] \cdot \widetilde{H}_{\lambda}.
\end{equation}
For example, if $\lambda = (3,2) \vdash 5$, we fill the boxes of $\lambda$ with monomials 
\begin{center}
$\begin{Young}
1 & q & q^2 \cr
t & qt
\end{Young}$
\end{center}
and see that
\begin{equation*}
\Delta_f(\widetilde{H}_{(3,2)}) = f(1,q,q^2,t,qt) \cdot \widetilde{H}_{(3,2)}.
\end{equation*}
When $f = e_n$, the restriction 
of the delta operator $\Delta_{e_n}$ to the space $\Lambda_n$
of homogeneous degree $n$ symmetric functions is more commonly denoted $\nabla$:
\begin{equation}
\Delta_{e_n} \mid_{\Lambda_n} = \nabla.
\end{equation}
In particular, we have $\Delta_{e_n} e_n = \nabla e_n$.

Given a partition $\lambda \vdash n$, let $S^{\lambda}$ denote the associated irreducible representation
of the symmetric group $S_n$; for example, we have that $S^{(n)}$ is the trivial representation
and $S^{(1^n)}$ is the sign representation.
Given any finite-dimensional $S_n$-module $V$, there exist unique integers $c_{\lambda}$ such that
$V \cong_{S_n} \bigoplus_{\lambda \vdash n} c_{\lambda} S^{\lambda}$.  The {\em Frobenius character}
of $V$ is the symmetric function
\begin{equation}
\Frob(V) := \sum_{\lambda \vdash n} c_{\lambda} \cdot s_{\lambda}
\end{equation}
obtained by replacing the irreducible $S^{\lambda}$ with the Schur function $s_{\lambda}$.

If $V = \bigoplus_{d \geq 0} V_d$ is a graded vector space, the {\em Hilbert series} of $V$
is the power series
\begin{equation}
\Hilb(V;q) = \sum_{d \geq 0} \dim(V_d) \cdot q^d.
\end{equation}
Similarly, if $V = \bigoplus_{d \geq 0} V_d$ is a graded $S_n$-module, the 
{\em graded Frobenius character} of $V$ is 
\begin{equation}
\grFrob(V;q) = \sum_{d \geq 0} \Frob(V_d) \cdot q^d.
\end{equation}

\subsection{Quasisymmetric and nonsymmetric functions}
The space $\Lambda$ of symmetric functions has many generalizations; in this paper
we will also use the spaces $\QSym$ of quasisymmetric functions and $\NSym$ of 
noncommutative symmetric 
functions.  We briefly review their definition below, as well as their relationship with the 0-Hecke algebra
$H_n(0)$; for more details see \cite{Huang, HR}.

Let $n$ be a positive integer.
A (strong) composition $\alpha$ of $n$ is a sequence $(\alpha_1, \dots, \alpha_k)$ of positive integers
with $\alpha_1 + \cdots + \alpha_k = n$.  We write $\alpha \models n$ or $|\alpha| = n$ to indicate
 that $\alpha$ is a composition of $n$.
 The map 
 $\alpha = (\alpha_1, \dots, \alpha_k) \mapsto \{\alpha_1, \alpha_1 + \alpha_2, \dots , \alpha_1 + \cdots + \alpha_{k-1} \}$
 provides a bijection between compositions of $n$ and subsets of $[n-1]$; we will find it convenient
 to identify compositions with subsets.

 Let $S \subseteq [n-1]$ be a subset.  The {\em Gessel fundamental quasisymmetric function} $F_S = F_{S,n}$
 attached to $S$ is the degree $n$ formal power series 
 \begin{equation}
 F_S = F_{S,n} := \sum_{\substack{i_1 \leq i_2 \leq \cdots \leq i_n \\ j \in S \, \, \Rightarrow \, \, i_j < i_{j+1}}}
 x_{i_1} \cdots x_{i_n}.
 \end{equation}
 The space $\QSym$ of {\em quasisymmetric functions} is the $\QQ(q,t)$-algebra of formal 
 power series with basis given by
 $\{ F_{S,n} \,:\, n \geq 0, \, S \subseteq [n-1] \}$.
 If a subset $S \subseteq [n-1]$ corresponds to a composition $\alpha$, we set $F_{\alpha} := F_{S,n}$.
 
 For any composition $\alpha \models n$, define a symbol $\sss_{\alpha}$
(the {\em noncommutative ribbon Schur function}), formally defined to have homogeneous degree $n$.
Let $\NSym_n$ be the $2^{n-1}$-dimensional
$\QQ(q,t)$-vector space with basis $\{ \sss_{\alpha} \,:\, \alpha \models n \}$ and let 
$\NSym$ be the graded vector space $\NSym := \bigoplus_{n \geq 0} \NSym_n$.
 The space $\NSym$ is the space of {\em noncommutative symmetric functions}.
 Although there is more structure on $\NSym$ (and on $\QSym$) than the graded vector space
 structure (namely, they are dual graded Hopf algebras), only the vector space structure will be relevant in this paper.
 
 Let $\FF$ be an arbitrary field.
The {\em 0-Hecke algebra} $H_n(0)$ of rank $n$ over $\FF$ is the unital associative  $\FF$-algebra with 
generators $T_1, \dots, T_{n-1}$ and relations
\begin{equation}
\begin{cases}
T_i^2 = T_i & 1 \leq i \leq n-1 \\
T_i T_j = T_j T_i & |i-j| > 1 \\
T_i T_{i+1} T_i = T_{i+1} T_i T_{i+1} & 1 \leq i \leq n-2.
\end{cases}
\end{equation}

For all $1 \leq i \leq n-1$, let $s_i := (i,i+1) \in S_n$ be the corresponding adjacent transposition.  
Given a permutation $\pi \in S_n$, we define $T_{\pi} := T_{i_1} \cdots T_{i_k} \in H_n(0)$, where
$\pi = s_{i_1} \cdots s_{i_k}$ is a reduced (i.e., short as possible) expression for $\pi$ 
as a product of adjacent transpositions.  It can be shown that 
$\{T_{\pi} \,:\, \pi \in S_n \}$ is a $\FF$-basis for $H_n(0)$, so that $H_n(0)$ has dimension $n!$ as a 
$\FF$-vector space and may be viewed as a deformation of the group algebra $\FF[S_n]$.
The algebra $H_n(0)$ is not semisimple, even when the field $\FF$ has characteristic zero, so
its representation theory has a different flavor from that of $S_n$.

The indecomposable projective  representations of $H_n(0)$ are naturally labeled by 
compositions $\alpha \models n$ (see \cite{Huang, HR}).  For $\alpha \models n$, we let $P_{\alpha}$
denote the corresponding indecomposable projective and let
\begin{equation}
C_{\alpha} := \mathrm{top}(P_{\alpha}) = P_{\alpha}/\mathrm{rad}(P_{\alpha})
\end{equation}
be the corresponding irreducible $H_n(0)$-module.

The  Grothendieck group $G_0(H_n(0))$ is the $\ZZ$-module generated by all isomorphism classes 
$[V]$ of finite-dimensional $H_n(0)$-modules with a relation
$[V] - [U] - [W] = 0$ for every short exact sequence $0 \rightarrow U \rightarrow V \rightarrow W \rightarrow 0$
of $H_n(0)$-modules.  The $\ZZ$-module $G_0(H_n(0))$ is free with basis given by (isomorphism classes of)
the irreducibles $\{C_{\alpha} \,:\, \alpha \models n \}$.  The {\em quasisymmetric characteristic map}
$\Ch$ is defined on $G_0(H_n(0))$ by 
\begin{equation}
\Ch: C_{\alpha} \mapsto F_{\alpha}.
\end{equation}
If a $H_n(0)$-module $V$ has composition factors $C_{\alpha^{(1)}}, \dots, C_{\alpha^{(k)}}$,
then  $\Ch(V) = F_{\alpha^{(1)}} + \cdots + F_{\alpha^{(k)}}$.
Since $H_n(0)$ is not semisimple, the characteristic $\Ch(V)$ does {\em not} determine $V$ up to isomorphism.

Let $K_0(H_n(0))$ be the $\ZZ$-module generated by all isomorphism classes $[P]$ 
of finite-dimensional {\em projective} $H_n(0)$-modules with a relation
$[P] - [Q] - [R] = 0$ for every short exact sequence $0 \rightarrow Q \rightarrow P \rightarrow R \rightarrow 0$ 
of projective modules.  The $\ZZ$-module $K_0(H_n(0))$ is free with basis
given by (isomorphism classes of) the projective indecomposable 
$\{P_{\alpha} \,:\, \alpha \models n\}$.  The {\em noncommutative characteristic map}
$\ch$ is defined on $K_0(H_n(0))$ by
\begin{equation}
\ch: P_{\alpha} \mapsto \sss_{\alpha},
\end{equation}
This extends to give a noncommutative symmetric function $\ch(P)$ for any projective $H_n(0)$-module $P$.
The module $P$ is determined by $\ch(P)$ up to isomorphism.

There are graded refinements of the maps $\Ch$ and $\ch$.  Let $V = \bigoplus_{d \geq 0} V_d$ 
be a graded $H_n(0)$-module with each $V_d$ finite-dimensional. 
The {\em degree-graded quasisymmetric characteristic} is 
$\Ch_q(V) := \sum_{d \geq 0} \Ch(V_d) \cdot q^d$.
If each $V_d$ is projective, the {\em degree-graded noncommutative characteristic}
is 
$\ch_q(V) := \sum_{d \geq 0} \ch(V_d) \cdot q^d$.

The quasisymmetric characteristic $\Ch$ admits a bigraded refinement as follows.  
The 0-Hecke algebra $H_n(0)$ has a {\em length filtration} 
\begin{equation}
H_n(0)^{(0)} \subseteq H_n(0)^{(1)} \subseteq H_n(0)^{(2)} \subseteq \cdots
\end{equation}
where $H_n(0)^{(\ell)}$ is the subspace of $H_n(0)$ with $\FF$-basis
$\{T_{\pi} \,:\, \pi \in S_n, \inv(\pi) \leq \ell \}$.  If $V = H_n(0)v$ is a cyclic $H_n(0)$-module
with distinguished generator $v$, we get an induced length filtration of $V$ by
\begin{equation}
V^{(\ell)} := H_n(0)^{(\ell)} v.
\end{equation}
The {\em length-graded quasisymmetric characteristic} is given by
\begin{equation}
\Ch_t(V) := \sum_{\ell \geq 0} \Ch(V^{(\ell)} / V^{(\ell-1)}) \cdot t^{\ell}.
\end{equation}

Now suppose $V = \bigoplus_{d \geq 0} V_d$ is a graded $H_n(0)$-module which is also cyclic.
We get a bifiltration of $V$ consisting of the modules $V^{(\ell,d)} := V^{(\ell)} \cap V_d$ for $\ell, d \geq 0$.
The {\em length-degree-bigraded quasisymmetric characteristic} is
\begin{equation}
\Ch_{q,t}(V) := \sum_{\ell, d \geq 0} \Ch(V^{(\ell,d)} / (V^{(\ell-1,d)} + V^{(\ell,d+1)})) \cdot q^d t^{\ell}.
\end{equation}
More generally, if $V$ is a direct sum of graded cyclic $H_n(0)$-modules, we define $\Ch_{q,t}(V)$ by 
applying $\Ch_{q,t}$ to its cyclic summands.  This may depend on the cyclic decomposition of the module $V$.
\footnote{Our conventions for $q$ and $t$ in the definitions of $\ch_q$ and $\Ch_{q,t}$ are reversed 
with respect to those in \cite{Huang, HR} and elsewhere.  We make these conventions so as to be consistent
with the case of the graded Frobenius map on $S_n$-modules.}

\subsection{Gr\"obner theory}
A total order $<$ on the monomials in the polynomial ring $\QQ[\xx_n]$ is called a {\em term order} if 
\begin{itemize}
\item  we have $1 \leq m$ for all monomials $m$, and
\item  $m \leq m'$ implies $m \cdot m'' \leq m' \cdot m''$ for all monomials $m, m', m''$.
\end{itemize}
The  term order used in this paper is the {\em lexicographic} term order given by
$x_1^{a_1} \cdots x_n^{a_n} < x_1^{b_1} \cdots x_n^{b_n}$ if there exists an index $i$
with $a_1 = b_1, \dots, a_{i-1} = b_{i-1}$ and $a_i < b_i$.

If $<$ is any term order any $f \in \QQ[\xx_n]$ is any nonzero polynomial, let $\initial_<(f)$ be the leading 
(i.e., greatest) term
of $f$ with respect to the order $<$.  If $I \subseteq \QQ[\xx_n]$ is any ideal, the associated {\em initial ideal}
is 
\begin{equation}
\initial_<(I) := \langle \initial_<(f) \,:\, f \in I - \{0\} \rangle.
\end{equation}

A finite collection $G = \{g_1, \dots, g_r \}$ of nonzero polynomials in an ideal $I \subseteq \QQ[\xx_n]$
 is called a {\em Gr\"obner basis}
of $I$ if we have the equality of monomial ideals
\begin{equation}
\initial_<(I) = \langle \initial_<(g_1), \dots, \initial_<(g_r) \rangle.
\end{equation}
If $G$ is a Gr\"obner basis of $I$ it follows that $I = \langle G \rangle$.  A Gr\"obner basis 
$G = \{g_1, \dots, g_r\}$ is called {\em minimal} if the $<$-leading coefficient of each $g_i$ is $1$ and
$\initial(g_i) \nmid \initial(g_j)$ for all $i \neq j$.  A minimal Gr\"obner basis $G = \{g_1, \dots, g_r\}$
is called {\em reduced} if in addition, for all $i \neq j$, no term of $g_j$ is divisible by $\initial_<(g_i)$.
After fixing a term order, every ideal $I \subseteq \QQ[\xx_n]$ has a unique reduced Gr\"obner basis.

Let $I \subseteq \QQ[\xx_n]$ be an ideal and let $G$ be a Gr\"obner basis for $I$.  
The set of monomials in $\QQ[\xx_n]$ 
\begin{equation}
\{m \,:\, \initial_<(f) \nmid m \text{ for all $f \in I - \{0\}$} \} = \{ m \,:\, \initial_<(g) \nmid m \text{ for all $g \in G$} \}
\end{equation}
descends to a vector space basis for the quotient $\QQ[\xx_n]/I$.  This is called the {\em standard monomial basis};
it is completely determined by the ideal $I$ and the term order $<$.  If $I$ is a homogeneous ideal, the Hilbert
series of $\QQ[\xx_n]/I$ is given by
\begin{equation}
\Hilb(\QQ[\xx_n]/I; q) = \sum_m q^{\deg(m)},
\end{equation}
where the sum is over all monomials in the standard monomial basis.

\section{Hilbert series}
\label{Hilbert}

In this section we will derive the Hilbert series and ungraded isomorphism type of the module 
$R_{n,k,r}$.  The method that we use dates back to Garsia and Procesi in the context of Tanisaki ideals and 
quotients \cite{GP}.

Let $Y \subseteq \QQ^n$ be any finite set of points and consider the ideal
$\II(Y) \subseteq \QQ[\xx_n]$ of polynomials which vanish on $Y$.  That is, we have
\begin{equation}
\II(Y) = \{ f \in \QQ[\xx_n] \,:\, f(\yy) = 0 \text{ for all $\yy \in X$} \}.
\end{equation}
We may identify the quotient $\QQ[\xx_n]/\II(Y)$ with the collection of all (polynomial) functions
$Y \rightarrow \QQ$; since $Y$ is finite we have
\begin{equation}
|Y| = \dim(\QQ[\xx_n]/\II(Y)).
\end{equation}
If $Y$ is stable under the coordinate permutation action of $S_n$, we have the further identification
of $S_n$-modules
\begin{equation}
\QQ[Y] \cong_{S_n}  \QQ[\xx_n]/\II(Y).
\end{equation}

The ideal $\II(Y)$ is usually not homogeneous; we wish to replace it by a homogeneous ideal so that the associated
quotient is graded.  For any nonzero polynomial $f \in \II(X)$, write $f = f_d + \cdots f_1 + f_0$ where $f_i$
is homogeneous of degree $i$ and $f_d \neq 0$.  Let $\tau(f) = f_d$ be the top homogeneous component of $f$.
The ideal $\TT(Y) \subseteq \QQ[\xx_n]$ is given by
\begin{equation}
\TT(Y) := \langle \tau(f) \,:\, f \in \II(Y) - \{0\} \rangle.
\end{equation}
By construction the ideal $\TT(Y)$ is homogeneous, so that the quotient $\QQ[\xx_n]/\TT(Y)$ is graded.
Furthermore, we still have the dimension equality
\begin{equation}
|Y| = \dim(\QQ[\xx_n]/\II(Y)) = \dim(\QQ[\xx_n]/\TT(Y))
\end{equation}
and the $S_n$-module isomorphism
\begin{equation}
\QQ[Y] \cong_{S_n} \QQ[\xx_n]/\II(Y) \cong_{S_n} \QQ[\xx_n]/\TT(Y)
\end{equation}
whenever the point set $Y$ is $S_n$-stable.

The symmetric group $S_n$ acts on $S_{n,k,r}$ by permuting the positive letters $1, 2, \dots, n$.
We aim to prove that $R_{n,k,r} \cong \QQ[S_{n,k,r}]$ as ungraded $S_n$-modules. To do this,
our strategy is as follows.
\begin{enumerate}
\item  Find a point set $Y_{n,k,r} \subseteq \QQ^n$ which is stable under the action of $S_n$ 
such that there is a $S_n$-equivariant bijection from $Y_{n,k,r}$ to $S_{n,k,r}$.
\item  Prove that $I_{n,k,r} \subseteq \TT(Y_{n,k,r})$ by showing that the generators of $I_{n,k,r}$ arise as top
degree components of polynomials in $\II(Y_{n,k,r})$.
\item  Prove that 
\begin{equation*}
\dim(R_{n,k,r}) = \dim(\QQ[\xx_n]/I_{n,k,r}) \leq |S_{n,k,r}| = \dim(\QQ[\xx_n]/\TT(Y_{n,k,r}))
\end{equation*}
and use the relation $I_{n,k,r} \subseteq \TT(Y_{n,k,r})$ to conclude that $I_{n,k,r} = \TT(Y_{n,k,r})$.
\end{enumerate}

The point set $Y_{n,k,r}$ which accomplishes Step 1 is the following.

\begin{defn}
\label{y-point-set}
Fix $n+k$ distinct rational numbers $\alpha_1, \alpha_2, \dots, \alpha_{n+k} \in \QQ$.  Let 
$Y_{n,k,r}$ be the set of points $(y_1, y_2, \dots, y_n) \in \QQ^n$ such that
\begin{itemize}
\item  the coordinates $y_1, y_2, \dots, y_n$ are distinct and lie in $\{\alpha_1, \alpha_2, \dots, \alpha_{n+k} \}$, and
\item the numbers $\alpha_{n+k-r+1}, \dots, \alpha_{n+k-1}, \alpha_{n+k}$ all appear as coordinates
$(y_1, y_2 \dots, y_n)$.
\end{itemize}
\end{defn}

It is clear that $Y_{n,k,r}$ is stable under the action of $S_n$.  We have a natural identification
of $Y_{n,k,r}$ with permutations in $S_{n,k,r}$ given by letting a copy of $\alpha_i$ in position $j$
of $(y_1, \dots, y_n)$ correspond to the letter $j$ in position $i$ of the corresponding permutation 
in $S_{n,k,r}$.  
For example, if $(n,k,r) = (4,3,2)$ then
\begin{equation*}
(\alpha_7, \alpha_2, \alpha_4, \alpha_6) \leftrightarrow  0203041.
\end{equation*}
This bijection $Y_{n,k,r} \leftrightarrow S_{n,k,r}$ is clearly $S_n$-equivariant,
so Step 1 of our strategy is accomplished.
Step 2 of our strategy is achieved in the following lemma.

\begin{lemma}
\label{j-contained-in-t}
We have $I_{n,k,r} \subseteq \TT(Y_{n,k,r})$.
\end{lemma}

\begin{proof}
We show that every generator of $I_{n,k,r}$ arises as the leading term of a polynomial in $\II(Y_{n,k,r})$.
We begin with the elementary symmetric function generators $e_{n-r+1}(\xx_n), \dots, e_{n-1}(\xx_n), e_n(\xx_n)$.
Consider the rational function in $t$ given by
\begin{equation}
\frac{(1 - x_1 t)(1 -  x_2 t) \cdots (1 -  x_n t)}{(1 - \alpha_{n+k-r+1} t) \cdots (1 - \alpha_{n+k-1} t) (1 - \alpha_{n+k} t)}
=  \sum_{i, j \geq 0} (-1)^i e_i(\xx_n) h_j(\alpha_{n+k-r+1}, \dots, \alpha_{n+k}) \cdot t^{i+j}.
\end{equation}
If $(x_1, \dots, x_n) \in Y_{n,k,r}$, the $r$  factors  in the denominator cancel  with $r$ factors in the numerator,
so that this rational expression is a polynomial in $t$ of degree $n-r$.  In particular, for $n-r+1 \leq m \leq r$ 
taking the coefficient
of $t^m$ on both sides gives
\begin{equation}
\sum_{i = 0}^m (-1)^i e_{m-i}(\xx_n) h_i(\alpha_{n+k-r+1}, \dots, \alpha_{n+k}) \in \II(Y_{n,k,r}),
\end{equation}
so that $e_m(\xx_n) \in \TT(Y_{n,k,r})$.

A similar trick shows that the homogeneous symmetric functions $h_{k+1}(\xx_n), \dots, h_{k+n}(\xx_n)$ lie
in $\TT(Y_{n,k,r})$.  Consider the rational function
\begin{equation}
\frac{(1 - \alpha_1 t)(1 - \alpha_2 t) \cdots (1 - \alpha_{n+k} t)}{(1 - x_1 t) (1 - x_2 t) \cdots (1 - x_n t)} =
\sum_{i,j \geq 0} (-1)^j  h_i(\xx_n)  e_j(\alpha_1, \dots, \alpha_{n+k}) t^{i+j}.
\end{equation}
If $(x_1, \dots, x_n) \in Y_{n,k,r}$ the $n$ factors in the denominator cancel with $n$ factors in the numerator,
giving a polynomial in $t$ of degree $k$.  For $m \geq k+1$, taking the coefficient of $t^m$ on both sides 
gives
\begin{equation}
\sum_{i = 0}^m (-1)^i h_{m-i}(\xx_n) e_i(\alpha_1, \dots, \alpha_{n+k}) \in \II(Y_{n,k,r}),
\end{equation}
so that $h_m(\xx_n) \in \TT(Y_{n,k,r})$.
\end{proof}

Step 3 of our strategy will take more work.  To begin, we identify a convenient collection of monomials 
in the initial ideal $\initial_<(I_{n,k,r})$ with respect to the lexicographic term order.  Given a subset 
$S = \{s_1 < \cdots < s_m\} = \subseteq [n]$ the corresponding {\em skip monomial}
$\xx(S)$ is given by
\begin{equation}
\xx(S) := x_{s_1}^{s_1} x_{s_2}^{s_2 - 1} \cdots x_{s_m}^{s_m-m+1}.
\end{equation}
In particular, if $n = 8$ we have $\xx(2458) = x_2^2 x_4^3 x_5^3 x_8^5$.

\begin{lemma}
\label{j-initial-ideal}
Let $<$ be the lexicographic term order on $\QQ[\xx_n]$.  If $S \subseteq [n]$ satisfies 
$|S| = n-r+1$ we have $\xx(S) \in \initial_<(I_{n,k,r})$.  Moreover, we have
$x_1^{k+1}, x_2^{k+2}, \dots, x_n^{k+n} \in \initial_<(I_{n,k,r})$.
\end{lemma}

\begin{proof}
The
assertion regarding skip monomials 
comes from combining \cite[Lem. 3.4]{HRS} (and in particular \cite[Eqn. 3.5]{HRS})
and \cite[Lem. 3.5]{HRS}.
To prove the second assertion, the identities 
\begin{equation}
h_{m+1}(x_i, x_{i+1}, \dots, x_n) - x_i \cdot h_m(x_i, x_{i+1}, \dots, x_n)  = h_{m+1}(x_{i+1}, \dots, x_n)
\end{equation}
(for $1 \leq i \leq n$ and $m \geq 0$) imply that
\begin{equation}
h_{k+1}(x_1, \dots, x_n), h_{k+2}(x_2, \dots, x_n), \dots, h_{k+n}(x_n) \in I_{n,k,r},
\end{equation}
so that $x_1^{k+1}, x_2^{k+2}, \dots, x_n^{k+n} \in \initial_<(I_{n,k,r})$.
\end{proof}

The initial terms provided by Lemma~\ref{j-initial-ideal} will be all we need.  We name the 
monomials $m \in \QQ[\xx_n]$ which are not divisible by any of these initial terms as follows.

\begin{defn}
\label{good-monomials}
A monomial $m \in \QQ[\xx_n]$ is {\em $(n,k,r)$-good} if 
\begin{itemize}
\item we have $\xx(S) \nmid m$ for all $S \subseteq [n]$ with $|S| = n-r+1$, and
\item we have $x_i^{k+i} \nmid m$ for all $1 \leq i \leq n$.
\end{itemize}
Let $\MMM_{n,k,r}$ denote the set of all $(n,k,r)$-good monomials.
\end{defn}

By Lemma~\ref{j-initial-ideal},
the monomials in $\MMM_{n,k,r}$ contain the standard monomial basis of $R_{n,k,r}$,
and so descend to a spanning set of $R_{n,k,r}$.  We will see that $\MMM_{n,k,r}$
is in fact that standard monomial basis of $R_{n,k,r}$.
We will do this using the following combinatorial result.

\begin{lemma}
\label{good-monomial-injection}
There is an injection $\Psi: S_{n,k,r} \rightarrow \MMM_{n,k,r}$ with the property that 
$\deg(\Psi(\pi)) = \inv(\pi)$ for all $\pi \in S_{n,k,r}$.
\end{lemma}

It will develop that the map $\Psi$ of Lemma~\ref{good-monomial-injection} is actually a bijection.

\begin{proof}
The map $\Psi$ will essentially be the inversion code.  Let $\pi = \pi_1 \dots \pi_{n+k} \in S_{n,k,r}$
be a $r$-tail positive permutation of the multiset $\{0^k, 1, 2, \dots, n\}$.  The {\em code} of $\pi$ is the sequence
$(c_1, \dots, c_n)$ where
\begin{equation}
c_i = \text{the number of letters $0, 1, 2, \dots, i-1$ to the right of $i$ in $\pi$}.
\end{equation}
For example, if $\pi = 40130052$ the code is $(c_1, c_2, c_3, c_4, c_5) = (2,0,3,6,1)$.  
It is clear that the sum of the code of $\pi$ gives the inversion number $\inv(\pi)$.  If $\pi \in S_{n,k,r}$ has 
code $(c_1, \dots, c_n)$, we define $\Psi(\pi) := x_1^{c_1} \cdots x_n^{c_n}$.

We argue that $\Psi$ is a well defined function $S_{n,k,r} \rightarrow \MMM_{n,k,r}$, that is, we have
$\Psi(\pi) \in \MMM_{n,k,r}$ for all $\pi \in S_{n,k,r}$.  Let $\pi \in S_{n,k,r}$ have code
$(c_1, \dots, c_n)$.  Since $\pi$ contains $k$ copies of $0$, it is clear that $c_i < k+i$ for all $1 \leq i \leq n$, so that
$x_i^{k+i} \nmid \Psi(\pi)$ for all $1 \leq i \leq n$.

Now let $S = \{s_1 < \cdots < s_{n-r+1} \} \subseteq [n]$ and suppose $\xx(S) \mid \Psi(\pi)$.
This means that $c_{s_i} \geq s_i - i + 1$ for all $1 \leq i \leq n-r+1$.  
Let $T = \{ \pi_{n+k-r+1}, \dots, \pi_{n+k-1}, \pi_{n+k} \}$ be the $r$-tail of $\pi$; since $\pi \in S_{n,k,r}$
the set $T$ consists of $r$ positive numbers.   We argue that $S \cap T = \emptyset$ as follows.
\begin{itemize}
\item
If $s_1 \in T$ we would have $c_{s_1} \leq s_1 - 1$ (since $s_1$ could form inversions with only 
$1, 2, \dots, s_1-1$), contradicting the inequality $c_{s_1} \geq s_1$.  We conclude that $s_1 \notin T$.
\item 
If $s_1, \dots, s_{i-1} \notin T$ and $s_i \in T$, we would have $c_{s_i} \leq s_i - i$ 
(since $s_i$ can only form inversions with those letters in $1, 2, \dots, s_i - 1$ which lie in $T$), 
contradicting the inequality $c_{s_i} \geq s_i - 1 + 1$.  We conclude that $s_i \notin T$.  
\end{itemize}
Induction gives the result that $S \cap T = \emptyset$.  However, this contradicts the facts that 
$|S| = n-r+1$, $|T| = r$, and that there are a total of $n$ positive letters in $\pi$.  This concludes the 
proof that the map $\Psi: S_{n,k,r} \rightarrow \MMM_{n,k,r}$ is well defined.

The relation $\deg(\Psi(\pi)) = \inv(\pi)$ is clear from construction.
The fact that $\Psi$ is an injection is equivalent to the fact that a permutation $\pi = \pi_1 \dots \pi_{n+k} \in S_{n,k,r}$
is determined by its code $(c_1, \dots, c_n)$.  This assertion is true more broadly
for any permutation of the multiset
$\{0^k, 1, 2, \dots, n\}$ (whether or not it is $r$-tail positive); we leave the verification to the reader.
\end{proof}

We are ready to derive the Hilbert series of $R_{n,k,r}$.

\begin{theorem}
\label{hilbert-series-theorem}
Endow monomials in $\QQ[\xx_n]$ with the lexicographic term order.  The standard monomial
basis of $R_{n,k,r}$ is $\MMM_{n,k,r}$.
The Hilbert series of $R_{n,k,r}$ is given by
\begin{equation}
\Hilb(R_{n,k,r}; q) = {n+k-r \brack k}_q \cdot [n]!_q.
\end{equation}
\end{theorem}

\begin{proof}
Let $\BBB^{\TT}$ be the standard monomial
basis of $\QQ[\xx_n]/\TT(Y_{n,k,r})$ and let $\BBB^J$ be the standard monomial basis of 
$R_{n,k,r} = \QQ[\xx_n]/I_{n,k,r}$.
We know that $|S_{n,k,r}| = |\BBB^{\TT}|$.
Lemma~\ref{j-contained-in-t} implies that $\BBB^{\TT} \subseteq \BBB^J$.
Lemma~\ref{j-initial-ideal} further implies the containment $\BBB^J \subseteq \MMM_{n,k,r}$.
Finally, Lemma~\ref{good-monomial-injection} gives the relation 
$|\MMM_{n,k,r}| \leq |S_{n,k,r}|$.  Putting these facts together gives 
\begin{equation}
\BBB^{\TT} = \BBB^J = \MMM_{n,k,r},
\end{equation}
and the fact that all of these sets have size $|S_{n,k,r}|$.
In particular, the standard monomial basis of $R_{n,k,r}$ is $\MMM_{n,k,r}$.

By the last paragraph, the map $\Psi$ of Lemma~\ref{good-monomial-injection} is a bijection.
It follows that 
\begin{equation}
\Hilb(R_{n,k,r};q) = \sum_{m \in \MMM_{n,k,r}} q^{\deg(m)} = \sum_{\pi \in S_{n,k,r}} q^{\inv(\pi)}.
\end{equation}
It is well known that $\sum_{\pi \in S_n} q^{\inv(\pi)} = [n]!_q$.  The $q$-binomial coefficient in
\begin{equation}
\sum_{\pi \in S_{n,k,r}} q^{\inv(\pi)} =  {n+k-r \brack k}_q \cdot [n]!_q
\end{equation}
comes from the ways of inserting $k$ copies of $0$ among the first $n-r$ letters of a permutation in 
$S_n$.
\end{proof}

We can also derive the ungraded $S_n$-isomorphism type of the quotient 
$R_{n,k,r}$.

\begin{corollary}
\label{ungraded-isomorphism-type}
As an {\em ungraded} $S_n$-module we have
$R_{n,k,r} \cong_{S_n} \QQ[S_{n,k,r}]$.
\end{corollary}

\begin{proof}
Lemma~\ref{j-contained-in-t} and Theorem~\ref{hilbert-series-theorem} give the isomorphisms
\begin{equation}
R_{n,k,r} \cong \QQ[\xx_n]/\TT(Y_{n,k,r}) \cong \QQ[\xx_n]/\II(Y_{n,k,r}) \cong \QQ[S_{n,k,r}]
\end{equation}
of ungraded $S_n$-modules.
\end{proof}

We describe a minimal Gr\"obner basis for the ideal $I_{n,k,r}$.  Given a subset 
$S = \{s_1 < \dots < s_m\} \subseteq [n]$, let $\gamma(S) = (\gamma(S)_1, \dots, \gamma(S)_n)$
be the length $n$ {\em skip vector} of nonnegative integers given by
\begin{equation}
\gamma(S)_i = \begin{cases}
s_j - j + 1 & i = s_j \\
0 & i \notin S.
\end{cases}
\end{equation}
Let $\gamma(S)^* = (\gamma(S)_n, \dots, \gamma(S)_1)$ be the reversal of the vector $\gamma(S)$.
If $\gamma = (\gamma_1, \dots, \gamma_n)$ is any length $n$ vector of nonnegative integers, let 
$\kappa_{\gamma}(\xx_n) \in \QQ[\xx_n]$ be the associated {\em Demazure character} 
(see \cite[Sec. 2]{HRS} for its definition).  Finally, if $f(\xx_n) \in \QQ[\xx_n]$ is any polynomial, let 
$f(\xx_n^*)$ be the polynomial obtained by reversing the variables in $f(\xx_n)$ so that
\begin{equation}
f(\xx_n^*) = f(x_n, x_{n-1}, \dots, x_1).
\end{equation}

\begin{corollary}
\label{groebner-basis-corollary}
Endow monomials in $\QQ[\xx_n]$ with the lexicographic term order.
A  Gr\"obner basis for the ideal $I_{n,k,r}$ consists of the polynomials
\begin{equation*}
h_{k+1}(x_1, x_2, \dots, x_n), h_{k+2}(x_2, \dots, x_n), \dots, h_{k+n}(x_n)
\end{equation*}
together with the polynomials
\begin{equation*}
\kappa_{\gamma(S)^*}(\xx_n^*),
\end{equation*}
where $S$ ranges over all $n-r+1$-element subsets of $[n]$. When $r < n$ and $k > 0$ this 
Gr\"obner basis is minimal.
\end{corollary}

The Gr\"obner basis in Corollary~\ref{groebner-basis-corollary} is typically not reduced.

\begin{proof}
The proof of Lemma~\ref{j-initial-ideal} shows that the polynomial
$h_{k+i}(x_i, x_{i+1}, \dots, x_n)$ lies in the ideal $I_{n,k,r}$.
By \cite[Lem. 3.4]{HRS} (and in particular \cite[Eqn. 3.5]{HRS}) shows that the 
relevant variable reversed Demazure characters lie in $I_{n,k,r}$.  

Let $<$ be the lexicographic term order on $\QQ[\xx_n]$.  We have
$\initial_<(h_{k+i}(x_i, x_{i+1}, \dots, x_n)) = x_i^{k+i}$ and
$\initial_<(\kappa_{\gamma(S)^*}(\xx_n^*) = \xx(S)$ (see \cite[Lem. 3.5]{HRS}).
We know that these initial terms generate $\initial_<(I_{n,k,r})$, proving the assertion about the 
claimed collection of polynomials being a Gr\"obner basis.  When $r < n$ and $k > 0$, none of the relevant
skip monomials $\xx(S)$ are divisible by any of the variable powers $x_1^{k+1}, \dots, x_n^{k+n}$.  This proves
 the claim about minimality.
\end{proof}

For example, consider the case $(n,k,r) = (5,2,3)$.  A minimal Gr\"obner basis for $J_{5,2,3}$ is given by
the polynomials
\begin{equation*}
h_3(x_1, x_2, x_3, x_4, x_5), h_4(x_2, x_3, x_4, x_5), h_5(x_3, x_4, x_5), h_6(x_4, x_5), h_7(x_5)
\end{equation*}
together with the variable reversed Demazure characters 
\begin{equation*}
\begin{array}{ccccc}
\kappa_{(0,0,1,1,1)}(\xx_5^*), &
\kappa_{(0,2,0,1,1)}(\xx_5^*), &
\kappa_{(3,0,0,1,1)}(\xx_5^*), &
\kappa_{(0,2,2,0,1)}(\xx_5^*), &
\kappa_{(3,0,2,0,1)}(\xx_5^*), \\
\kappa_{(3,3,0,0,1)}(\xx_5^*), &
\kappa_{(0,2,2,2,0)}(\xx_5^*), &
\kappa_{(3,0,2,2,0)}(\xx_5^*), &
\kappa_{(3,3,0,2,0)}(\xx_5^*), &
\kappa_{(3,3,3,0,0)}(\xx_5^*).
\end{array}
\end{equation*}

Theorem~\ref{hilbert-series-theorem} describes the standard monomial basis
$\MMM_{n,k,r}$ of $R_{n,k,r}$ in terms of divisibility by skip monomials.  However, a more direct characterization
of this standard monomial basis is available.  
Let $k \geq 0$ and $r \leq n$. For any $(n-r)$-element subset $T \subseteq [n]$, define a length
$n$ sequence $\delta(T) := (\delta(T)_1, \dots, \delta(T)_n)$ by the formula
\begin{equation}
\delta(T)_i := \begin{cases}
i + k - 1 & i \in T \\
j - 1 & i \notin T \text{ and } i = s_j,
\end{cases}
\end{equation}
where $[n] - T = \{s_1 < \cdots < s_r\}$.
Any of the ${n \choose r}$ sequences which can be obtained in this way is an {\em $(n,k,r)$-staircase}.
For example, the $(5,2,3)$-staircases are
\begin{equation*}
\begin{array}{ccccc}
(0,1,2,5,6), &
(0,1,4,2,6), &
(0,3,1,2,6), &
(2,0,1,2,6), &
(0,1,4,5,2), \\
(0,3,1,5,2), &
(2,0,1,5,2), &
(0,3,4,1,2), &
(2,0,4,1,2), &
(2,3,0,1,2).
\end{array}
\end{equation*}

\begin{proposition}
\label{standard-monomial-basis}
Endow monomials in $\QQ[\xx_n]$ with the lexicographic term order. 
The standard monomial basis $\MMM_{n,k,r}$ of $R_{n,k,r}$ consists of
those monomials in $\QQ[\xx_n]$ whose exponent vectors are componentwise $\leq$
at least one $(n,k,r)$-staircase.
\end{proposition}

\begin{proof}
Let $\NNN_{n,k,r}$ be the collection of monomials in $\QQ[\xx_n]$ whose exponent
vectors are componentwise $\leq$ at least one $(n,k,r)$-staircase.  
If $\delta_n(T) = (a_1, \dots, a_n)$ is an $(n,k,r)$-staircase 
for some $(n-r)$-element set $T \subseteq [n]$ and $m = x_1^{a_1} \cdots x_n^{a_n}$ is the 
corresponding monomial, 
it is clear that $a_i < k+i$ for all $i$, so that 
$x_i^{k+i} \nmid m$.  
If $S \subseteq [n]$ satisfies $|S| = n-r+1$ then at least one index $i \in S$ satisfies 
$i \notin T$, which forces $\xx(S) \nmid m$.  It follows that $\NNN_{n,k,r} \subseteq \MMM_{n,k,r}$.

On the other hand, we may construct a map
\begin{equation}
\Phi: S_{n,k,r} \rightarrow \NNN_{n,k,r}
\end{equation}
by letting $\Phi(\pi) = (c_1, \dots, c_n)$ be the code of any $r$-tail positive permutation $\pi \in S_{n,k,r}$.
The fact that $\pi$ is $r$-tail positive implies that $\Phi(\pi) \in \NNN_{n,k,r}$, so that $\Phi$
is well defined.  It is clear that $\Phi$ is injective, so that 
\begin{equation}
|S_{n,k,r}| \leq |\NNN_{n,k,r}| \leq |\MMM_{n,k,r}| = |S_{n,k,r}|
\end{equation}
and we have $\NNN_{n,k,r} = \MMM_{n,k,r}$, as desired.
\end{proof}

For example, if $(n,k,r) = (2,2,1)$ the $(2,2,1)$-staircases are $(0,3)$ and $(2,0)$ so that 
\begin{equation*}
\MMM_{2,2,1} = \{1, x_1, x_1^2, x_2, x_2^2, x_2^3 \}.
\end{equation*}

\section{Frobenius series}
\label{Frobenius}

In this section we derive the Frobenius series of $R_{n,k,r}$.  Our first lemma is a short exact sequence
which establishes a Pascal-type recursion for $\grFrob(R_{n,k,r}; q)$.

\begin{lemma}
\label{short-exact-sequence}
Suppose $n, k, r \geq 0$ with $r < n$ and $k > 0$.
There is a short exact sequence of $S_n$-modules
\begin{equation}
0 \rightarrow R_{n,k-1,r} \rightarrow R_{n,k,r} \rightarrow R_{n,k,r+1} \rightarrow 0,
\end{equation}
where the first map is homogeneous of degree $n-r$ and the second map is homogeneous of degree $0$.  
Equivalently, we have the equality of graded Frobenius characters
\begin{equation}
\grFrob(R_{n,k,r}; q) = \grFrob(R_{n,k,r+1}; q) + q^{n-r} \cdot \grFrob(R_{n,k-1,r}; q).
\end{equation}
\end{lemma}

\begin{proof}
We have the inclusion of ideals $I_{n,k,r} \subseteq I_{n,k,r+1}$; we let the second map be the canonical
projection $\pi: R_{n,k,r} \twoheadrightarrow R_{n,k,r+1}$.  
We have a homogeneous map $\widetilde{\varphi}: \QQ[\xx_n] \rightarrow R_{n,k,r}$ of degree $n-r$
given by multiplication by $e_{n-r}(\xx_n)$, and then projecting onto $R_{n,k,r}$.  

We claim that 
$\widetilde{\varphi}(I_{n,k-1,r}) = 0$, so that $\widetilde{\varphi}$ induces a well defined map
$\varphi: R_{n,k-1,r} \rightarrow R_{n,k,r}$.  This is equivalent to showing that 
$h_k(\xx_n) \cdot e_{n-r}(\xx_n) \in I_{n,k,r}$.  The Pieri Rule implies that 
\begin{equation}
h_k(\xx_n) \cdot e_{n-r}(\xx_n) = s_{(k,1^{n-r})}(\xx_n) + s_{(k+1, 1^{n-r-1})}(\xx_n);
\end{equation}
we will show that both terms on the right hand side lie in $I_{n,k,r}$.

To see that $s_{(k,1^{n-r})}(\xx_n) \in I_{n,k,r}$, observe that, for $1 \leq i \leq r$ we have
\begin{equation}
\label{crucial-relation}
h_{k-r+i}(\xx_n) \cdot e_{n-i+1}(\xx_n) = s_{(k-r+i,1^{n-i+1})}(\xx_n) + s_{(k-r+i+1,1^{n-i})}(\xx_n) \in I_{n,k,r}.
\end{equation}
It follows that modulo $I_{n,k,r}$ we have the congruences
\begin{equation}
\label{congruence-chain}
s_{(k,1^{n-r})}(\xx_n)  \equiv - s_{(k+1,1^{n-r-1})}(\xx_n) \equiv s_{(k+2,1^{n-r-2})}(\xx_n) \equiv \cdots 
\equiv \pm s_{(k+n-r)}(\xx_n) \equiv 0,
\end{equation}
where the last congruence used the fact that $s_{(k+n-r)}(\xx_n) = h_{k+n-r}(\xx_n) \in I_{n,k,r}$ since $r < n$.
This chain of congruences also shows that $s_{(k+1,1^{n-r-1})}(\xx_n) \in I_{n,k,r}$.

By the last paragraph, we have a well defined induced map $\varphi: R_{n,k-1,r} \rightarrow R_{n,k,r}$.  It is clear that
$\mathrm{Im}(\varphi) = \mathrm{Ker}(\pi)$.  Moreover, the Pascal relation implies that 
\begin{equation}
|S_{n,k-1,r}| + |S_{n,k,r+1}| = |S_{n,k,r}|,
\end{equation}
so that by Theorem~\ref{hilbert-series-theorem} we have
\begin{equation}
\dim(R_{n,k-1,r}) + \dim(R_{n,k,r+1}) = \dim(R_{n,k,r}).
\end{equation}
Since $\pi$ is a surjection, this forces the sequence 
\begin{equation}
0 \rightarrow R_{n,k-1,r} \xrightarrow{\varphi} R_{n,k,r} \xrightarrow{\pi} R_{n,k,r+1} \rightarrow 0
\end{equation}
to be exact.  To finish the proof,  observe that the maps $\varphi$ and $\pi$ commute with the action
of $S_n$.
\end{proof}

We are ready to state the graded Frobenius image of $R_{n,k,r}$.  We will give several formulas for this image.
For any word $w$ over the nonnegative integers, define the monomial $\xx^w$ to be
\begin{equation}
\xx^w := x_1^{\text{\# of $1$'s in $w$}} x_2^{\text{\# of $2$'s in $w$}} \cdots;
\end{equation}
in particular, any copies of $0$ in $w$ do not affect $\xx^w$.

\begin{theorem}
\label{graded-isomorphism-type}
The graded Frobenius image of $R_{n,k,r}$ is given by
\begin{align}
\grFrob(R_{n,k,r};q) 
&= {n+k-r \brack k}_q \cdot \sum_{T \in \SYT(n)} q^{\mathrm{maj}(T)} \cdot s_{\mathrm{shape}(T)} \\
&= \sum_w q^{\inv(w)} \xx^w.
\end{align}
The last sum ranges over all length $n+k$ words $w = w_1 \dots w_{n+k}$ in the alphabet of nonnegative integers
which contain precisely $k$ copies of $0$ and are $r$-tail positive.  
\end{theorem}

\begin{proof}
By considering the placement of the $k$ copies of $0$ in a $r$-tail positive word $w$ appearing in the final sum, 
we see that
\begin{equation}
\sum_w q^{\inv(w)} \xx^w = {n+k-r \brack k}_q \cdot
\sum_{\substack{v = v_1 \dots v_n \\ v_i \in \ZZ_{> 0}}} q^{\inv(v)} \xx^v.
\end{equation}
On the other hand, we have
\begin{equation}
\sum_{\substack{v = v_1 \dots v_n \\ v_i \in \ZZ_{> 0}}} q^{\inv(v)} \xx^v = 
\sum_{\substack{v = v_1 \dots v_n \\ v_i \in \ZZ_{> 0}}} q^{\maj(v)} \xx^v =
\sum_{T \in \SYT(n)} q^{\mathrm{maj}(T)} \cdot s_{\mathrm{shape}(T)} =
\grFrob(R_n;q),
\end{equation}
where the first equality uses the equidistribution of the statistics $\inv$ and $\maj$ on permutations of a fixed
multiset of positive integer, the second follows from standard properties of the RSK correspondence,
and the third is a consequence of the work of Lusztig-Stanley \cite{Stanley}.

By the last paragraph, it suffices to prove the first equality asserted in the statement of the theorem.
If $r \geq n$ or $k = 0$ then $R_{n,k,r} = R_n$ and this equality is trivial.  Otherwise, we have the $q$-Pascal relation
\begin{equation}
{n+k-r \brack k}_q = {n+k-r-1 \brack k}_q + q^{n-r} \cdot {n+k-r-1 \brack k-1}_q,
\end{equation}
so that the theorem follows from Lemma~\ref{short-exact-sequence} and induction.
\end{proof}

The short exact sequence in Lemma~\ref{short-exact-sequence} gives a recipe for constructing 
bases of $R_{n,k,r}$ from bases of the classical coinvariant algebra $R_n$.  
We switch from working over $\QQ$ to working over an arbitrary field $\FF$, so that 
the ideals $I_{n,k,r}, I_n$ are defined inside the ring $\FF[\xx_n] := \FF[x_1, \dots, x_n]$
and we have $R_{n,k,r} := \FF[\xx_n]/I_{n,k,r}, R_n := \FF[\xx_n]/I_n$.

\begin{theorem}
\label{basis-recipe}
Let $\CCC_n = \{b_{\pi}(\xx_n) \,:\, \pi \in S_n\}$ be a collection of polynomials in $\FF[\xx_n]$ indexed by permutations
in $S_n$ which descends to a basis of $R_n$.  The collection of polynomials 
\begin{equation} \CCC_{n,k,r} :=
\{ b_{\pi}(\xx_n) \cdot e_{\lambda}(\xx_n) \,:\, \text{$\pi \in S_n$, $\lambda_1 \leq n-r$, 
and $\lambda$ has $\leq k$ parts} \}
\end{equation}
in $\FF[\xx_n]$
descends to a basis of $R_{n,k,r}$.
\end{theorem}

\begin{proof}
This is trivial when $k = 0$ or $r \geq n$, so we assume $k > 0$ and $r < n$.  

The arguments of Section~\ref{Hilbert} apply to show that $\dim(R_{n,k,r}) = |S_{n,k,r}|$ when working 
over the arbitrary field $\FF$.
\footnote{If $\FF$ is a finite field, there might not be enough elements in $\FF$ for the 
point set $Y_{n,k,r}$ of Definition~\ref{y-point-set} to make sense.  To get around this, we may apply
\cite[Lem. 3.1]{HR} to harmlessly replace $\FF$ by an extension field $\mathbb{K}$.}
The proof of Lemma~\ref{short-exact-sequence} then applies over $\FF$
to give a short exact sequence of graded $\FF$-vector spaces
\begin{equation}
0 \rightarrow R_{n,k-1,r} \xrightarrow{\cdot e_{n-r}(\xx_n)} R_{n,k,r} \xrightarrow{\pi} R_{n,k,r+1} \rightarrow 0,
\end{equation}
where $\pi$ is the canonical projection.  
We may inductively assume that $\CCC_{n,k-1,r}$ descends to a $\FF$-basis of $R_{n,k-1,r}$ and that 
$\CCC_{n,k,r+1}$ descends to a $\FF$-basis of $S_{n,k,r+1}$.  Exactness implies that  
\begin{equation}
\{ f(\xx_n) \,:\, f(\xx_n) \in \CCC_{n,k,r+1} \} \cup \{ g(\xx_n) \cdot e_{n-r}(\xx_n) \,:\, g(\xx_n) \in \CCC_{n,k-1,r} \} =
\CCC_{n,k,r}
\end{equation}
descends to an $\FF$-basis of $R_{n,k,r}$.
\end{proof}

Theorem~\ref{basis-recipe} reinforces the fact that $R_{n,k,r}$ consists of 
${n+k-r \choose k}$ copies of $R_n$, graded by the $q$-binomial coefficient ${n+k-r \brack k}_q$.
Interesting bases $\CCC_n$ to which Theorem~\ref{basis-recipe} can be applied include
\begin{itemize}
\item  the {\em Artin basis} \cite{Artin}
\begin{equation}
\CCC_n = \{ x_1^{i_1} \cdots x_n^{i_n} \,:\, 0 \leq i_j < j \}
\end{equation}
(which is connected to the $\inv$ statistic on permutations in $S_n$) and
\item the {\em Garsia-Stanton basis} (or the {\em descent monomial basis} \cite{Garsia, GS}
$\CCC_n = \{gs_{\pi} \,:\, \pi \in S_n \}$ where 
\begin{equation}
gs_{\pi} = \prod_{\pi_i > \pi_{i+1}} x_{\pi_1} \cdots x_{\pi_i}
\end{equation}
(which is connected to the $\maj$ statistic on permutations in $S_n$).
\end{itemize}

The GS basis above can be deformed somewhat to describe the isomorphism type of $R_{n,k,r}$
as a module over the 0-Hecke algebra.
The algebra $H_n(0)$ acts on the polynomial ring $\FF[\xx_n]$ by letting the generator $T_i$ act by 
the {\em Demazure operator} $\sigma_i$, where
\begin{equation}
\sigma_i . f := \frac{x_i f - x_{i+1} s_i(f)}{x_i - x_{i+1}}.
\end{equation}
Here $s_i(f)$ is the polynomial obtained by interchanging $x_i$ and $x_{i+1}$ in $f(\xx_n)$.
It can be shown that if $f \in \FF[\xx_n]^{S_n}$ is any symmetric polynomial and $g \in \FF[\xx_n]$
is an arbitrary polynomial then
\begin{equation}
\sigma_i(fg) = f \sigma_i(g).
\end{equation}
Therefore,
any ideal $I \subseteq \FF[\xx_n]$ generated by symmetric polynomials is stable under
the action of $H_n(0)$.  In particular, the ideal $I_{n,k,r}$ is stable under the action of $H_n(0)$,
and the quotient $R_{n,k,r} = \FF[\xx_n]/I_{n,k,r}$ carries the structure of an $H_n(0)$-module.

Huang \cite{Huang} studied the coinvariant ring $R_n$ as a graded module over the 0-Hecke algebra
$H_n(0)$.
We apply Theorem~\ref{basis-recipe} to generalize Huang's results to the quotient $R_{n,k,r}$.
If $V$ is any graded $H_n(0)$-module, we let $V(i)$ denote the graded $H_n(0)$-module
with components $V(i)_j := V_{i+j}$.

\begin{theorem}
\label{zero-hecke-theorem}
Let $n,k,$ and $r$ be nonnegative integers with $r \leq n$.  We have an isomorphism of 
graded $H_n(0)$-modules
\begin{equation}
\label{direct-sum-equation}
R_{n,k,r} \cong \bigoplus_{\lambda \subseteq (n-r) \times k} R_n(-|\lambda|).
\end{equation}
Here the direct sum is over all partitions $\lambda$ which satisfy $\lambda_1 \leq n-r$ and have at most $k$ 
parts.  The module $R_n = \FF[\xx_n]/I_n$ is the coinvariant algebra viewed as a graded $H_n(0)$-module.
\end{theorem}

\begin{proof}
Huang \cite{Huang} introduced the following modified GS 
basis of $R_n$.  For $1 \leq i \leq n-1$, define an operator $\bar{\sigma}_i$ on $\FF[\xx_n]$ by the rule
$\bar{\sigma}_i := \sigma_i - 1$.  For any permutation $\pi \in S_n$, define 
$\bar{\sigma}_{\pi} := \bar{\sigma}_{i_1} \cdots \bar{\sigma}_{i_k}$ where $s_{i_1} \cdots s_{i_k}$
is any reduced word for $\pi$.  Finally, given $\pi \in S_n$, let $\xx_{\Des(\pi)}$ be the monomial
\begin{equation}
\xx_{\Des(\pi)} := \prod_{i \in \Des(\pi)} (x_1 x_2 \cdots x_i).
\end{equation}
For example, we have $\xx_{21543} = (x_1) \cdot (x_1 x_2 x_3) \cdot (x_1 x_2 x_3 x_4)$.
Huang proves \cite[Thm. 4.5]{Huang} that the collection of polynomials
\begin{equation}
\CCC_n := \{ \bar{\sigma}_{\pi}. \xx_{\Des(\pi)} \,:\, \pi \in S_n \}
\end{equation}
in $\FF[\xx_n]$ descends to a $\FF$-basis for $R_n$.

Applying Theorem~\ref{basis-recipe} to Huang's basis of $R_n$,
we get a collection of polynomials $\CCC_{n,k,r}$ given by
\begin{equation}
\CCC_{n,k,r} := \{ e_{\lambda}(\xx_n) \cdot \bar{\sigma}_{\pi}.\xx_{\Des(\pi)} \,:\, 
\pi \in S_n \text{ and $\lambda \subseteq (n-r) \times k$} \}
\end{equation}
which descends to a basis of $R_{n,k,r}$.
The symmetric polynomial $e_{\lambda}(\xx_n)$ has degree $|\lambda|$ and the 
symmetry of $e_{\lambda}(\xx_n)$ gives
\begin{equation}
\sigma_i.(e_{\lambda}(\xx_n) \cdot \bar{\sigma}_{\pi}.\xx_{\Des(\pi)}) = 
e_{\lambda}(\xx_n) \sigma_i.(\bar{\sigma}_{\pi}.\xx_{\Des(\pi)}).
\end{equation}
It follows that, for $\lambda \subseteq (n-r) \times k$ fixed, the collection of polynomials
\begin{equation}
\CCC_{n,k,r}(\lambda) := \{ e_{\lambda}(\xx_n) \cdot \bar{\sigma}_{\pi}.\xx_{\Des(\pi)} \,:\, 
\pi \in S_n  \}
\end{equation}
descends inside $R_{n,k,r}$ to a $\FF$-basis of a copy of $R_n$ with degree shifted up by $|\lambda|$.
\end{proof}

It may be tempting to try to prove Theorem~\ref{zero-hecke-theorem} in the same fashion 
as Theorem~\ref{graded-isomorphism-type} -- by applying the short exact sequence of 
Lemma~\ref{short-exact-sequence} directly and without appealing to Theorem~\ref{basis-recipe}.
However, 
although the maps in this sequence commute with the action of $H_n(0)$,
since  $H_n(0)$ is not semisimple
it is not {\em a priori} clear that this sequence
splits in the category of $H_n(0)$-modules.  
Theorem~\ref{zero-hecke-theorem} guarantees that this 
sequence splits; the authors do not know of a more direct way to see this splitting.

\begin{corollary}
\label{characteristic-hecke}
Let $n, k,$ and $r$ be nonnegative integers with $r \leq n$.
\begin{enumerate}
\item The length-degree bigraded quasisymmetric characteristic $\mathrm{Ch}_{q,t}(R_{n,k,r})$ is given by
\begin{equation}
\mathrm{Ch}_{q,t}(R_{n,k,r}) = {n+k-r \brack k}_q  \cdot \mathrm{Ch}_{q,t}(R_n) =
{n+k-r \brack k}_q \cdot \sum_{\pi \in S_n} q^{\maj(\pi)} t^{\inv(\pi)} 
F_{\Des(\pi^{-1}),n},
\end{equation}
where $F_{\Des(\pi^{-1}),n}$ is the fundamental quasisymmetric function.
\item  The degree graded quasisymmetric characteristic $\mathrm{Ch}_q(R_{n,k,r})$ is in fact symmetric
and given by
\begin{equation}
\mathrm{Ch}_q(R_{n,k,r}) = {n+k-r \brack k}_q  \cdot \mathrm{Ch}_{q}(R_n) =
{n+k-r \brack k}_q \cdot \sum_{T \in \SYT(n)} q^{\maj(T)} s_{\mathrm{shape}(T)}.
\end{equation}
\item  The $H_n(0)$-module $R_{n,k,r}$ is projective.  Its degree
 graded noncommutative characteristic $\mathbf{ch}_q(R_{n,k,r})$ is
\begin{equation}
\mathbf{ch}_q(R_{n,k,r}) = {n+k-r \brack k}_q  \cdot \mathbf{ch}_{q}(R_n) =
{n+k-r \brack k}_q  \cdot \sum_{\alpha} q^{\maj(\alpha)} \mathbf{s}_{\alpha},
\end{equation}
where $\alpha$ ranges over all strong compositions of $n$, the major index is 
$\maj(\alpha) = (\alpha_1) + (\alpha_1 + \alpha_2) + \cdots$, and 
$\mathbf{s}_{\alpha}$ is the noncommutative ribbon Schur function.
\end{enumerate}
\end{corollary}

\begin{proof}
Parts 1 and 2 follow from the work of
Huang \cite[Cor. 4.9]{Huang} and Theorem~\ref{zero-hecke-theorem}.  
Since $R_n$ is a projective $H_n(0)$-module (see \cite[Thm. 4.5]{Huang})
and direct sums of projective modules are projective, we can apply \cite[Cor. 8.4, $\mu = (1^n)$]{Huang} to
get Part 3.
\end{proof}

Since the characteristics $\mathrm{Ch}_{q,t}$ and $\mathrm{Ch}_t$ are defined in terms of 
the Grothendieck group $G_0(H_n(0))$
of $H_n(0)$, we may apply the short exact sequence of Lemma~\ref{short-exact-sequence}
to obtain Parts 1 and 2 of Corollary~\ref{characteristic-hecke} more directly.
However, since extensions of projective modules are not in general projective, Lemma~\ref{short-exact-sequence}
does not immediately imply that $R_{n,k,r}$ is a projective $H_n(0)$-module.

Although Theorem~\ref{basis-recipe} gives a collection of {\em polynomials} in $\FF[\xx_n]$
generalizing the GS monomials
which descend to a basis of $R_{n,k,r}$, the authors have been unable to find a collection of
{\em monomials} in $\FF[\xx_n]$ which generalizes the GS monomials and descends to a 
basis of $R_{n,k,r}$ (such monomial bases were found for the quotients appearing in the work
of Haglund-Rhoades-Shimozono and Huang-Rhoades \cite{HRS, HR}).
Judging from the construction in \cite[Sec. 5]{HRS} and the Hilbert series of $R_{n,k,r}$, one might expect that
the set of monomials
\begin{equation}
\{ gs_{\pi} \cdot x_{\pi_1}^{i_1} \cdots x_{\pi_{n-r}}^{i_{n-r}} \,:\, \pi \in S_n \text{ and } 
k \geq i_1 \geq \cdots \geq i_{n-r} \geq 0 \}
\end{equation}
would descend to a basis of $R_{n,k,r}$, but this set of monomials is linearly dependent in the quotient in general.
A potential combinatorial obstruction to finding a GS monomial basis for $R_{n,k,r}$ is the fact that the statistics
$\inv$ and $\maj$ do {\em not} share the same distribution on $S_{n,k,r}$.

\section{Open problems}
\label{Open}

\subsection{Bivariate generalization for $r=1$}

We propose a relationship between our quotient ring $R_{n,k,r}$ and the theory of Macdonald 
polynomials.  In particular, consider the ideal 
$I'_{n,k,r} \subseteq \QQ[\xx_n]$ given  by 
\begin{equation}
I'_{n,k,r} := \langle p_{k+1}(\xx_n), p_{k+2}(\xx_n), \dots, p_{k+n}(\xx_n), e_n(\xx_n), e_{n-1}(\xx_n), \dots, 
e_{n-r+1}(\xx_n) \rangle
\end{equation}
and let $R'_{n,k,r} := \QQ[\xx_n]/I_{n,k,r}$ be the corresponding quotient.  
The ideal $I'_{n,k,r}$ is obtained from the ideal $I_{n,k,r}$ by replacing the homogeneous
symmetric functions with power sum symmetric functions.

As with the quotient $R_{n,k,r}$,
the quotient $R'_{n,k,r}$ has the structure of a graded $S_n$-module.  
Although the ideals $I_{n,k,r}$ and $I'_{n,k,r}$ are not equal in general, we present

\begin{conjecture}
\label{module-isomorphism-conjecture}
There is an isomorphism of graded $S_n$-modules $R_{n,k,r} \cong R'_{n,k,r}$.
\end{conjecture}

The main reason for preferring the quotient rings $R'_{n,k,r}$ over the quotient rings $R_{n,k,r}$ is that they 
generalize more readily to two sets of variables.  Let $\xx_n = (x_1, \dots, x_n)$ and $\yy_n = (y_1, \dots, y_n)$
be two sets of $n$ variables and let $\QQ[\xx_n, \yy_n]$ be the polynomial ring in these variables.  The symmetric
group $S_n$ acts on $\QQ[\xx_n,\yy_n]$ by the {\em diagonal action}
$\pi.x_i = x_{\pi_i}, \pi.y_i = y_{\pi_i}$.

For any $a,b \geq 0$, let $p_{a,b}(\xx_n,\yy_n)$ be the {\em polarized power sum}
\begin{equation}
p_{a,b}(\xx_n,\yy_n) := \sum_{i = 1}^n x_i^a y_i^b.
\end{equation}
Moreover, let $\MMM_n$ be the set of the $2^n$ monomials $z_1 \dots z_n$ in $\QQ[\xx_n,\yy_n]$ where
$z_i \in \{x_i, y_i\}$ for all $1 \leq i \leq n$.  For example, we have
\begin{equation*}
\MMM_2 = \{x_1 x_2, x_1 y_2, y_1 x_2, y_1 y_2 \}.
\end{equation*}

For a nonnegative integer $k$,
let $DI_{n,k} \subseteq \QQ[\xx_n,\yy_n]$ be the ideal generated by the polarized power sums
$p_{a,b}(\xx_n,\yy_n)$ with $a+b \geq k+1$ together with the monomials in $\MMM_n$.  
Let $DR_{n,k} := \QQ[\xx_n,\yy_n]/DI_{n,k}$ be the corresponding quotient, which is a bigraded $S_n$-module.

%

\begin{conjecture}
\label{delta-operator-conjecture}
The bigraded Frobenius image of $DR_{n,k}$ is given by the delta operator image
\begin{equation*}
\grFrob(DR_{n,k};q,t) = \Delta_{h_k e_n} e_n = \Delta_{s_{k+1, 1^{n-1}}} e_n = \Delta_{h_k} \nabla e_n.
\end{equation*}
\end{conjecture}

The latter three quantities in the conjecture are trivially equal by the definition of the delta operator.
When $k = 0$, the ring $DR_{n,0}$ is the classical diagonal coinvariant ring $DR_n$, so that 
Conjecture~\ref{delta-operator-conjecture} reduces to Haiman's celebrated result \cite{Haiman}
that $\grFrob(DR_n) = \Delta_{e_n} e_n$.   Setting the $\yy_n$ variables equal to zero in the quotient
$DR_{n,k}$ yields the ring $R'_{n,k,1}$, so that the ring $R'_{n,k,1}$ conjecturally gives
 the analog of the coinvariant ring (for one set of variables) attached to the operator 
 $\Delta_{h_k e_n}$.
 
 The following proposition states that our module $R_{n,k,1}$ has graded Frobenius series which
 agrees with any of the delta operator expressions in Conjecture~\ref{delta-operator-conjecture} upon
 setting $q = 0$ and $t = q$.
 
 \begin{proposition}
 We have 
 \begin{equation*}
 \grFrob(R_{n,k,1};t) = \Delta_{h_k e_n} e_n \mid_{q = 0} = \Delta_{s_{k-1,1^{n-1}}} e_n \mid_{q = 0}
 = \Delta_{h_k} \nabla e_n \mid_{q = 0}.
 \end{equation*}
 \end{proposition}
 
 \begin{proof}
 In this proof we will use the notation of plethysm; we refer the reader to \cite{Hagbook} for the relevant 
 details on plethysm and symmetric functions.
 
 Let $\rev_t$ be the operator which reverses the coefficient sequences of polynomials with respect to the variable $t$.
 For a partition $\lambda \vdash n$, let $Q'_{\lambda} = Q'_{\lambda}(\xx;t)$ 
 be the corresponding Hall-Littlewood symmetric function.  
 It is well known that the modified Macdonald polynomial
 $\widetilde{H}_{\lambda} = \widetilde{H}_{\lambda}(\xx;q,t)$ satisfies
 \begin{equation}
 \widetilde{H}_{\lambda} \mid_{q = 0} = \rev_t (Q'_{\lambda}).
 \end{equation}
 This means that, for any symmetric function $f$ and any partition $\lambda \vdash n$, we have
 \begin{equation}
 \label{eigen-equation}
 \Delta_f (\widetilde{H}_{\lambda}) \mid_{q = 0} 
 = f(1,t,t^2, \dots, t^{\ell(\lambda)-1}) \cdot \rev_t(Q'_{\lambda}),
 \end{equation}
 where $\ell(\lambda)$ is the number of parts of $\lambda$.  
 
 In order to exploit Equation~\ref{eigen-equation}, we need to express $e_n$ in terms of the modified Macdonald basis.
 This expansion is found in \cite[Eqn. 2.72]{Hagbook}: we have
 \begin{equation}
 \label{elementary-expansion}
 e_n =  \sum_{\lambda \vdash n} \frac{M B_{\lambda} \Pi_{\lambda} \widetilde{H}_{\lambda}}{w_{\lambda}},
 \end{equation}
 where 
 \begin{itemize}
 \item $M = (1-q)(1-t)$,
 \item  $B_{\lambda} = \sum_{c = (i,j) \in \lambda}  q^{i-1} t^{j-i}$, where the sum is over all 
 cells $c$ with matrix coordinates $(i,j)$ in the Ferrers diagram of $\lambda$,
 \item $\Pi_{\lambda} = \prod_{c = (i,j) \neq (0,0)} (1 - q^{i-1} t^{j - 1})$, where $c = (i,j)$ is a cell in $\lambda$
 other than the  corner $(0,0)$,
 \item  $w_{\lambda} = \prod_{c \in \lambda} (q^{a(c)} - t^{l(c) + 1})(t^{l(c)} - q^{a(c) + 1})$,
 where the product is over all cells $c$ in the Ferrers diagram of $\lambda$ and $a(c), l(c)$ denote the arm 
 and leg lengths of $\lambda$ at $c$.
 \end{itemize}

We apply the operator $\Delta_{h_k e_n} = \Delta_{h_k} \Delta_{e_n}$ to both sides of 
Equation~\ref{elementary-expansion} to get 
\begin{equation}
\label{first-equation}
 \Delta_{h_k e_n} e_n =  
 \sum_{\lambda \vdash n} h_k[B_{\lambda}] e_n[B_{\lambda}] 
 \frac{M B_{\lambda} \Pi_{\lambda} \widetilde{H}_{\lambda}}{w_{\lambda}}.
\end{equation}
Setting $q = 0$ on both sides of Equation~\ref{first-equation} gives
\begin{equation}
\label{second-equation}
 \Delta_{h_k e_n} e_n \mid_{q = 0} =  
\left[ \sum_{\lambda \vdash n} h_k[B_{\lambda}] e_n[B_{\lambda}] 
 \frac{M B_{\lambda} \Pi_{\lambda} \widetilde{H}_{\lambda}}{w_{\lambda}} \right]_{q = 0}.
\end{equation}

For any $\lambda \vdash n$ and any symmetric function $f$, we have
$f[B_{\lambda}] \mid_{q = 0} = f(1,t,t^2, \dots, t^{\ell(\lambda)-1})$.  In particular, we have
$e_n[B_{\lambda}] = 0$ unless $\lambda = (1^n)$ and Equation~\ref{second-equation} reduces to
\begin{equation}
\label{third-equation}
\Delta_{h_k e_n} e_n \mid_{q = 0} =  h_k(1,t,\dots,t^{n-1}) \cdot
e_n(1,t,\dots, t^{n-1}) \cdot
\left[  \frac{M B_{(1^n)} \Pi_{(1^n)} \widetilde{H}_{(1^n)}}{w_{(1^n)}} \right]_{q = 0}.
\end{equation}
The right hand side of Equation~\ref{third-equation} simplifies to
\begin{equation}
{n + k - 1 \brack k}_t \cdot \rev_t(Q'_{(1^n)}) 
= \grFrob(R_{n,k,1};t)
\end{equation}
where we used Theorem~\ref{graded-isomorphism-type} at $r = 1$ and 
the well known fact that the graded Frobenius image of the classical coinvariant algebra
$R_n$ is $\grFrob(R_n;t) = \rev_t(Q'_{(1^n)})$.
 \end{proof}

 \subsection{Other bivariate generalizations}
 
One may wonder if there is a bivariate generalization of the entire ring $R_{n,k,r}$, as we have only discussed the $r=1$ case so far. While we have not been able to find a full generalization, there is some progress in the Hilbert series case. The \emph{skewing operator} acts on a symmetric function $f$ of degree $d$ uniquely so that
 \begin{equation}
 \langle \partial f , g \rangle = \langle f, p_1 g \rangle
 \end{equation}
 for all symmetric functions $g$ of degree $d-1$, where the inner product is the usual Hall inner product on symmetric functions. 
 
 Given a vector $\alpha = (\alpha_1, \dots, \alpha_n)$ 
 of $n$ positive integers, an \emph{$\alpha$-Tesler matrix} $U = (u_{i,j})_{1 \leq i, j \leq n}$ 
 is an $n \times n$ upper 
 triangular matrix with nonnegative integer entries such that, for $i=1$ to $n$,
 \begin{equation}
 u_{i, i} + u_{i, i+1} + \ldots + u_{i, n} - (u_{1, i} + u_{2, i} + \ldots + u_{i-1, i}) = \alpha_i .
 \end{equation}
 We write $U \in \mathcal{T}(\alpha)$.
 For example, the matrix
 \begin{equation*}
 U = \begin{pmatrix}
 0 & 2 & 0 & 1 \\
 0 & 3 & 1 & 0 \\
 0 & 0 & 1 & 2 \\
 0 & 0 & 0 & 6
 \end{pmatrix}
 \end{equation*}
 satisfies $U \in \mathcal{T}(3,2,2,3)$. 
 The \emph{weight} of an $n \times n$ $\alpha$-Tesler matrix $U$ is equal to
 \begin{equation}
 \operatorname{wt}(U; q ,t) = (-(1-q)(1-t))^{\operatorname{pos}(U) - n} \prod_{u_{i,j} > 0} [u_{i,j}]_{q, t}
 \end{equation}
 where $\operatorname{pos}(U)$ is the number of positive entries in $U$ and $[k]_{q,t}$ is the usual $q,t$-integer, i.e.\ $[k]_{q,t} = \frac{q^k - t^k}{q-t}$. 
 For example, if $U$ is the Tesler matrix shown above, we have $\operatorname{pos}(U) = 7$ and
 \begin{equation*}
 \operatorname{wt}(U;q,t) = (-(1-q)(1-t))^{7-4} [2]_{q,t}^2 [3]_{q,t} [6]_{q,t}.
 \end{equation*}
 Finally,  the \emph{$\alpha$-Tesler polynomial} is
 \begin{equation}
 \operatorname{Tes}(\alpha; q, t) = \sum_{U \in \mathcal{T}(\alpha) } \wt(U; q, t) .
 \end{equation}
 This corollary follows from work in \cite{AGHRS, Haglund, LMM}.

\begin{corollary}
\begin{align}
\Hilb(R_{n,k,r} ;q) &= {n + k - r \brack k}_q \cdot [n]!_q \\
&= \left. \partial^{n-r+1} \Delta_{h_k} \partial^{r-1} \nabla e_n \right|_{t=0} \\
&= \sum_{\substack{\alpha \vDash n+k \\ \ell(\alpha) = n \\ \alpha_1 = \ldots = \alpha_r = 1}} \operatorname{Tes}(\alpha; q, 0)
\end{align}
\end{corollary}

It would be interesting to find an extension of this corollary to the entire graded Frobenius series of $R_{n,k,r}$ for general $r$. 



\subsection{A Schubert basis}
\def\perms{\Pi}
\def\schubert{\operatorname{Schubert}}
\def\lt{\mathcal{LT}}
There is also a basis for $R_{n,k,r}$ given by certain Schubert polynomials. We let $\perms_{n,k,r}$ be all the permutations $\pi$ of $\{1,2,\ldots,n+k\}$ that satisfy
\begin{itemize}
\item all descents in $\pi$ occur weakly left of position $n$, and
\item $1,2,\ldots,r$ all appear in $\pi_1 \pi_2 \ldots \pi_n$. 
\end{itemize}
If $\pi \in \Pi_{n,k,r}$ is  a permutation, let $\symm_{\pi}(\xx_n)$ be the Schubert polynomial attached to $\pi$.
Note that, since each $\pi$ has no descents after position $n$, there are at most $n$ variables that appear in the Schubert polynomial associated to $\pi$, so we have not truncated the variable set in any meaningful way.
We will show that $\{\symm_\pi(\xx_n^*) : \pi \in \perms_{n,k,r}\}$ is a basis for $R_{n,k,r}$, where the asterisk represents the reversal of the vector of variables. 
This will follow from the fact that the leading terms are all $(n,k,r)$-good monomials.

\begin{proposition}
Let $<$ be the lexicographic monomial order and
let 
\begin{align*}
\lt_{n,k,r} &= \{ \initial_<(\symm_{\pi}(\xx_n^*)) : \pi \in \perms_{n,k,r} \}.
\end{align*}
Then $\lt_{n,k,r} = \MMM_{n,k,r}$.
\end{proposition}

\begin{proof}
We will construct a bijection $\Phi : \perms_{n,k,r} \to \MMM_{n,k,r}$ that satisfies 
$\Phi(\pi) = \initial_<(\symm_{\pi}(\xx_n^*))$. The bijection itself is
\begin{align}
\Phi(\pi) &= \prod_{i=1}^{n} x_{n-i+1} ^{d_i}
\end{align} 
where $d_i$ counts the number of $j > i$ such that $\pi_i > \pi_j$.
The fact that $\Phi(\pi) = \initial_<(\symm_{\pi}(\xx_n^*))$ follows directly from the definition of the Schubert polynomial. 
We need to show that $m = \Phi(\pi) \in \MMM_{n,k,r}$ and to construct its inverse. Our proof will be similar to that of Lemma \ref{good-monomial-injection}. First, we check that 
\begin{itemize}
\item we have $\xx(S) \nmid m$ for all $S \subseteq [n]$ with $|S| = n-r+1$, and
\item we have $x_i^{k+i} \nmid m$ for all $1 \leq i \leq n$.
\end{itemize}
To check the first condition, we recall that $\xx(S) = x_{s_1}^{s_1} x_{s_2}^{s_2 - 1} \ldots x_{s_{n-r+1}}^{s_{n-r+1} - n + r}$ if $S = \{s_1 < s_2 < \ldots < s_{n-r+1}\}$. Since $S \subseteq [n]$ and the entries $1$ through $r$ all appear in $\pi_1$ through $\pi_n$, there is some $s_i$ such that $\pi_{s_i} \leq r$. Choose $i$ as large as possible such that $\pi_{s_i} \leq r$. Since $j > n$ implies $\pi_j > r$, $\pi_{s_i}$ can only be greater than at most $n-s_i$ entries to its right, i.e.\ $d_{s_i} \leq n-s_i$. Hence the power of $x_{n-s_i+1}$ in $m$ is at most $n-s_i$, which means $\xx(S) \nmid m$. The second condition follows from the definition of $m$.

Given a monomial $m \in \MMM_{n,k,r}$, we would like to construct $\Phi^{-1}(m)$. This can be done using the usual bijection from codes $(d_1, d_2, \ldots, d_n)$ to permutations. For $i=1$ to $n$, we choose $\pi_i$ such that it is greater than exactly $d_{i}$ of the entries in $[n+k]$ that have not already been placed to the left of position $i$ in $\pi$. The second condition for $(n,k,r)$-good monomials implies that the result is an honest permutation, and the first condition implies that $1,2,\ldots,r$ all appear in the first $n$ entries.
\end{proof}

\begin{corollary}
$\{\symm_\pi(\xx_n^*) : \pi \in \perms_{n,k,r}\}$ descends to a basis for $R_{n,k,r}$.
\end{corollary}

It would be interesting to explore if this Schubert basis maintains many of the properties of the Schubert basis for the usual ring of coinvariants. For example, the following suggests that the structure constants of this Schubert basis are positive modulo $R_{n,k,r}$. 

\begin{question}
For two permutations $\pi, \pi^{\prime} \in \perms_{n,k,r}$, is it always true that the product
\begin{align}
\symm_{\pi}(\xx_n^*) \times \symm_{\pi^{\prime}}(\xx_n^{*})
\end{align}
has positive integer coefficients when expanded in the 
basis $\{\symm_\pi(\xx_n^*) : \pi \in \perms_{n,k,r}\}$ modulo $I_{n,k,r}$? 
Using \textsc{Sage}, we have checked that this is true for 
$1 \leq n, k \leq 4$ and $0 \leq r \leq n$. If so, do these coefficients count intersections in some family of varieties?
\end{question}


\section{Acknowledgements}
\label{Acknowledgements}
 
B. Rhoades was partially supported by the NSF Grant DMS-1500838.
A. T. Wilson was partially supported by a NSF Graduate Sciences Postdoctoral Research Fellowship.

\end{document}